\documentclass[12pt]{amsart}
\usepackage{amssymb,mathtools,amsthm}

\newtheorem{thm}{Theorem}[section]
\newtheorem{prop}[thm]{Proposition}
\newtheorem{cor}[thm]{Corollary}
\theoremstyle{definition}
\newtheorem{defn}[thm]{Definition}
\newtheorem{rem}[thm]{Remark}

\newcommand{\bk}{{\boldsymbol{k}}}
\newcommand{\bn}{{\boldsymbol{n}}}
\newcommand{\bv}{{\boldsymbol{v}}}
\newcommand{\bx}{{\boldsymbol{x}}}
\newcommand{\by}{{\boldsymbol{y}}}
\newcommand{\bz}{{\boldsymbol{z}}}
\newcommand{\bN}{{\boldsymbol{N}}}

\DeclareMathOperator{\dep}{dep}

\title[Monotonicity of Coefficients of Interpolated MZVs]
{Monotonicity of Coefficients of Interpolated Multiple Zeta Values}
\author{Shuji Yamamoto}
\address{Interdisciplinary Faculty of Science and Engineering, Shimane University, 
1060 Nishi-Kawatsu, Matsue, 690-8504, Japan. }
\email{yamashu@riko.shimane-u.ac.jp}

\thanks{This research was supported by JSPS KAKENHI JP21K03185.}

\begin{document}

\begin{abstract}
Let $\zeta^t(\bk)=\sum_{d=1}^{r}S_d(\bk)t^{r-d}$ be 
the interpolated multiple zeta value of 
an admissible index $\bk$ of depth $r$. 
We show the monotonicity property of coefficients $S_1(\bk)\ge\cdots\ge S_r(\bk)>0$. 
In the proof, we represent the coefficients as integrals over a common measure space 
and compares the integrands pointwise. 
As consequences, we determine the sign and vanishing of $\zeta^{-1}(\bk)$ 
and locate the zeros of $\zeta^t(\bk)$ in the closed unit disk.
\end{abstract}

\maketitle

\section{Introduction}
For an admissible index $\bk=(k_1,\ldots,k_r)$, 
the interpolated multiple zeta value is defined by 
\begin{align*}
\zeta^t(\bk)
&\coloneqq \sum_{\bk'\preceq\bk}t^{r-\dep(\bk')}\zeta(\bk')
=\sum_{0<n_1\le\cdots\le n_r}
\frac{t^{\#\{i\mid n_i=n_{i+1}\}}}{n_1^{k_1}\dots n_r^{k_r}},
\end{align*}
where $\bk'\preceq\bk$ means that an index $\bk'$ is obtained 
by changing some commas in $\bk=(k_1,\ldots,k_r)$ to plus symbols. 
For example, 
\[\zeta^t(k_1,k_2,k_3)=\zeta(k_1,k_2,k_3)+t\zeta(k_1+k_2,k_3)
+t\zeta(k_1,k_2+k_3)+t^2\zeta(k_1+k_2+k_3). \]
In fact, the variables $k_1,\ldots,k_r$ are not necessarily integers; 
in what follows, we fix an $r$-tuple $\bk=(k_1,\ldots,k_r)$ of \emph{real numbers} 
with $k_1,\ldots,k_{r-1}\ge 1$, $k_r\ge 2$. 

To state our main theorem, we write $S_d(\bk)$ 
for the coefficient of $t^{r-d}$ in the polynomial $\zeta^t(\bk)$. 
Explicitly, we have 
\[S_d(\bk)=\sum_{\substack{\bk'\preceq\bk\\ \dep(\bk')=d}}
\zeta(\bk'). \]

\begin{thm}\label{thm:main}
For each integer $d$ with $1\le d<r$, we have  
\begin{equation}\label{eq:main}
S_d(\bk)\ge S_{d+1}(\bk). 
\end{equation}
Equality holds if and only if $\bk=(\{1\}^{r-1},2)$. 
\end{thm}

We have the following consequences. 

\begin{cor}\label{cor:t=-1}
We have 
\[(-1)^{r-1}\zeta^{-1}(\bk)\ge 0. \]
Equality holds if and only if $r$ is even and $\bk=(\{1\}^{r-1},2)$. 
\end{cor}
\begin{proof}
Since 
\[(-1)^{r-1}\zeta^{-1}(\bk)=\begin{cases}
(S_1-S_2)+\dots+(S_{r-1}-S_r) & (\text{$r$:even}), \\
(S_1-S_2)+\dots+(S_{r-2}-S_{r-1})+S_r & (\text{$r$:odd}), 
\end{cases}\]
the claim follows from Theorem \ref{thm:main}. 
\end{proof}

\begin{cor}\label{cor:roots}
All roots $\alpha$ of the polynomial $\zeta^t(\bk)$ satisfy $|\alpha|\le 1$. 
Equality holds for some root if and only if $r\ge 2$ and $\bk=(\{1\}^{r-1},2)$. 
\end{cor}
\begin{proof}
This is an application of the Enestr\"{o}m--Kakeya theorem \cite{K}. 
Note that, if $\bk\ne(\{1\}^{r-1},2)$, then 
$\max_{1\le d<r}S_{d+1}(\bk)/S_d(\bk)<1$ by Theorem \ref{thm:main}. 
\end{proof}

To prove Theorem \ref{thm:main}, 
we will construct certain measure spaces $\Omega^{(d)}$ 
and functions $\bN^{(d)}$ on them, 
whose integrals represent the values $S_d(\bk)$ (\S2). 
Then we show that $\Omega^{(d)}$ is essentially isomorphic 
to a single measure space $\Omega$ independent of $d=1,\dots,r$ (\S3), 
and prove the inequality between $\bN^{(d)}$ and $\bN^{(d+1)}$ 
which implies $S_d(\bk)\ge S_{d+1}(\bk)$ (\S4). 

The interpolated multiple zeta values ($t$-MZVs for short) 
were first introduced in \cite{Y} and 
have been investigated by several authors (e.g.\ \cite{TW,LC,L}), 
mainly from the viewpoint of their algebraic structure. 
Our result has a different flavor from these earlier works 
since it considers \emph{inequality} rather than \emph{equality}. 

On the other hand, we also remark that there is some possibility 
of proving Theorem \ref{thm:main} in a more algebraic way. 
In fact, for admissible indices $\bk$, 
one may even expect that $S_d(\bk)-S_{d+1}(\bk)$ can be 
represented as linear combinations of MZVs 
with \emph{nonnegative integer} coefficients. 
Numerical experiments with indices of small weight support this conjecture. 

\subsection*{Use of AI}
A large part of this paper emerged from discussions with ChatGPT. 
In particular, the initial version of the proof of Theorem \ref{thm:main} was generated by ChatGPT-6 Astra. 
Then the author rewrote it entirely while preserving the basic idea. 
The author takes full responsibility for the content of the paper.

\section{Integral representation of $S_d(\bk)$}
Let $d$ and $r$ be integers with $1\le d\le r$. 
In this section, we will construct a certain measure space 
$(\Omega^{(d)},\mu^{(d)})$ 
and an $r$-tuple $\bN=(N_1,\ldots,N_r)$ of functions 
on $\Omega^{(d)}$ taking values in the positive integers 
(we regard $r$ as fixed throughout, and omit it from the notation). 
Then we will show: 

\begin{prop}\label{prop:mu N}
We have 
\begin{equation}\label{eq:mu N}
\mu^{(d)}\bigl(\bN=\bn\bigr)
=\begin{cases}
\dfrac{1}{n_1\dots n_{r-1}n_r^2} & (n_1\le\dots\le n_r
\text{ and } \rho(\bn)=d), \\
0 & (\text{otherwise}) 
\end{cases}
\end{equation}
for any $r$-tuple $\bn=(n_1,\ldots,n_r)$ of positive integers, 
where $\rho(\bn)$ denotes the number of distinct components in $\bn$, 
i.e., $\rho(\bn)=\#\{n_1,\ldots,n_r\}$. 
\end{prop}

The equation \eqref{eq:mu N} implies the following expression 
of $S_d(\bk)$: 
\begin{equation}\label{eq:S_d integral}
S_d(\bk)=\int_{\Omega^{(d)}}
N_1^{-(k_1-1)}\dots N_{r-1}^{-(k_{r-1}-1)}
N_r^{-(k_r-2)}\mu^{(d)}. 
\end{equation}
This representation is the key to our proof of Theorem \ref{thm:main}. 

We first define a space $\Delta^{(d)}$ by 
\[\Delta^{(d)}=
\biggl\{(x_0,\dots,x_r)\in (0,1)^{r+1}\biggm|
\begin{tabular}{l}
$x_0<x_1<\dots<x_{d-1}<x_d$, \\
$x_0<x_r<\dots<x_{d+1}<x_d$
\end{tabular}\biggr\}. \]
This space is equipped with the measure 
\[\prod_{i=0}^{d-1}\frac{dx_i}{1-x_i}\cdot\prod_{i=d}^r\frac{dx_i}{x_i}. \]
The volume of $\Delta^{(d)}$ with respect to this measure is 
\begin{equation}\label{eq:vol Delta(d)}
\int_{\Delta^{(d)}}
\prod_{i=0}^{d-1}\frac{dx_i}{1-x_i}\cdot\prod_{i=d}^r\frac{dx_i}{x_i}
=S_d(\{1\}^{r-1},2); 
\end{equation}
indeed, each shuffle of $(x_1,\ldots,x_{d-1})$ and 
$(x_r,\ldots,x_{d+1})$ gives the iterated integral expression of 
a multiple zeta value of weight $r+1$, depth $d$.  
For later use, we describe this ``decomposition by shuffle'' in more detail. 

In what follows, we often disregard the null set 
on which two coordinates $x_i$ and $x_j$ coincide. 
For each $\bx=(x_0,\ldots,x_r)\in\Delta^{(d)}$, we set 
\[p_i=p_i(\bx)\coloneqq \#\{j\in\{0,\ldots,r\}\mid x_j<x_i\}
\quad (0\le i\le d) \]
and 
\[J(\bx)\coloneqq \{p_1(\bx),\ldots,p_{d-1}(\bx)\}. \]
Then we have 
\[0=p_0<p_1<\cdots<p_{d-1}<p_d=r. \]
If we set 
\[\Delta^{(d)}_J\coloneqq\{\bx\in\Delta^{(d)}\mid J(\bx)=J\}\]
for each $J\subset\{1,\ldots,r-1\}$ with $\#J=d-1$, 
we obtain the disjoint decomposition 
\begin{equation}\label{eq:J decomp}
\Delta^{(d)}=\coprod_{J}\Delta^{(d)}_J. 
\end{equation}
For each such $J=\{p_1,\ldots,p_{d-1}\}$, 
let us use another set of coordinates $\by=(y_0,\ldots,y_r)$ 
on $\Delta^{(d)}_J$ which is obtained by arranging $(x_0,\ldots,x_r)$ 
in increasing order; 
thus we have $y_{p_i}=x_i$ for $0\le i\le d$, and the other $y_j$'s 
correspond to $x_r,\ldots,x_{d+1}$. 
Then 
\begin{align*}
&\int_{\Delta^{(d)}_J}
\prod_{i=0}^{d-1}\frac{dx_i}{1-x_i}\cdot\prod_{i=d}^r\frac{dx_i}{x_i}\\
&=\int_{0<y_0<\cdots<y_r<1}
\prod_{i=0}^{d-1}\biggl(\frac{dy_{p_i}}{1-y_{p_i}}
\prod_{j=p_i+1}^{p_{i+1}-1}\frac{dy_j}{y_j}\biggr)
\cdot \frac{dy_r}{y_r}\\
&=\zeta(p_1,p_2-p_1,\ldots,p_{d-1}-p_{d-2},r-p_{d-1}+1). 
\end{align*}
By summing over all possible $J$'s, 
we obtain the equation \eqref{eq:vol Delta(d)}. 

Next we construct the functions $N_1,\ldots,N_r$. 
For this purpose, we need to extend the space to 
$\Omega^{(d)}\coloneqq \Delta^{(d)}\times(0,1)^r$. 
The coordinates on the latter component $(0,1)^r$ will be denoted 
by $\bv=(v_0,\dots,v_{r-1})$. 
We equip this space $\Omega^{(d)}$ with the measure 
\[\mu^{(d)}\coloneqq\prod_{i=0}^{d-1}\frac{dx_i}{1-x_i}
\cdot\prod_{i=d}^r\frac{dx_i}{x_i}
\cdot\prod_{j=0}^{r-1}dv_j. \]
Thus the volume is still $\mu^{(d)}(\Omega^{(d)})=S_d(\{1\}^{r-1},2)$. 

\begin{defn}
For each subset $J$ of $\{1,\dots,r-1\}$ with $\#J=d-1$, 
we define functions $\delta_j$ ($j=0,\dots,r-1$) 
and $N_l$ ($l=1,\ldots,r$) on 
$\Omega^{(d)}_J\coloneqq\Delta^{(d)}_J\times(0,1)^r$ by 
\begin{align*}
\delta_j(\by,\bv)&\coloneqq \begin{cases}
\lceil\log_{x_i}v_i\rceil & (j=p_i,\ 0\le i<d), \\
0 & (\text{otherwise}), 
\end{cases}\\
N_l(\by,\bv)&\coloneqq\sum_{j=0}^{l-1}\delta_j(\by,\bv), 
\end{align*}
using the coordinates $\by=(y_0,\dots,y_r)$ on $\Delta^{(d)}_J$ and 
$\bv=(v_0,\dots,v_{r-1})$ on $(0,1)^r$, respectively. 
Then we collect them together to form the functions $\delta_j$ 
and $N_l$ on $\Omega^{(d)}=\coprod_{J}\Omega^{(d)}_J$. 
\end{defn}

\begin{proof}[Proof of Proposition \ref{prop:mu N}]
Take any $r$-tuple $\bn=(n_1,\dots,n_r)$ of positive integers and put $n_0=0$. 
Then the equation $\bN=\bn$ holds only if $\bn$ satisfies 
\[n_1\le\dots\le n_r\text{ and }\rho(\bn)=d. \]
If this is the case, 
the equation $\bN=\bn$ can occur only on $\Omega^{(d)}_J$ with 
\[J=\{1\le j\le r-1\mid n_j<n_{j+1}\},\]
and is equivalent to the conditions 
\[y_j^{n_{j+1}-n_j}\le v_i<y_j^{n_{j+1}-n_j-1}\quad 
(j=p_i\in\{0\}\cup J,\ 0\le i<d). \]
Therefore, we have 
\begin{align*}
&\mu^{(d)}(\bN=\bn)\\
&=\int_{0<y_0<\dots<y_r<1}
\prod_{j\in \{0\}\cup J}(y_j^{n_{j+1}-n_j-1}-y_j^{n_{j+1}-n_j})\frac{dy_j}{1-y_j}
\cdot\prod_{j\in\{1,\dots,r\}\setminus J}\frac{dy_j}{y_j}\\
&=\int_{0<y_0<\dots<y_r<1}
\prod_{j=0}^{r-1}\bigl(y_j^{n_{j+1}-n_j-1}dy_j\bigr)\cdot \frac{dy_r}{y_r}\\
&=\frac{1}{n_1\dots n_{r-1}n_r^2}. 
\end{align*}
This completes the proof of Proposition \ref{prop:mu N}. 
\end{proof}

\section{Unification of spaces}
In this section, we construct an explicit isomorphism of 
measure spaces between $\Omega\coloneqq\Omega^{(1)}$ and $\Omega^{(d)}$. 

The coordinates on $\Delta\coloneqq\Delta^{(1)}$ and $\Delta^{(d)}$ 
will be written as $\bz=(z_0,\dots,z_r)$ and $\bx=(x_0,\dots,x_r)$; 
by definition, they satisfy 
\begin{align*}
&0<z_0<z_r<\dots<z_2<z_1<1,\\
&0<x_0<x_1<\dots<x_{d-1}<x_d<1, \\
&0<x_0<x_r<\dots<x_{d+1}<x_d<1. 
\end{align*}
For notational convenience, we set $z_{r+1}=z_0$ and $x_{r+1}=x_0$. 

We consider the map from $\Delta^{(d)}$ to $\Delta$ given by 
\begin{equation}\label{eq:x to z}
z_i\coloneqq\begin{cases}
\dfrac{1-x_i}{1-x_0} & (1\le i\le d), \\
\dfrac{x_i(1-x_d)}{x_d(1-x_0)} & (d<i\le r+1). 
\end{cases}
\end{equation}
It is easily verified that this is a bijection. 
Moreover, this map preserves the measure: 

\begin{prop}\label{prop:measures}
Under the transformation \eqref{eq:x to z}, 
we have the identity of measures
\[\frac{dz_0}{1-z_0}\prod_{i=1}^r\frac{dz_i}{z_i}
=\prod_{i=0}^{d-1}\frac{dx_i}{1-x_i}\cdot\prod_{i=d}^r\frac{dx_i}{x_i}. \] 
\end{prop}
\begin{proof}
First consider the transformation 
\[\begin{cases}
z_0=\dfrac{x_0(1-x_d)}{x_d(1-x_0)}\\
z_d=\dfrac{1-x_d}{1-x_0}
\end{cases}\]
from $(x_0,x_d)$ to $(z_0,z_d)$. A computation of the Jacobian shows 
\[\frac{dz_0}{1-z_0}\frac{dz_d}{z_d}=\frac{dx_0}{1-x_0}\frac{dx_d}{x_d}. \]
Moreover, once these variables are fixed, the remaining variables are 
transformed linearly, and one easily shows that 
\[\frac{dz_i}{z_i}=\begin{cases}
\dfrac{dx_i}{1-x_i} & (0<i<d),\\
\dfrac{dx_i}{x_i} & (d<i\le r) 
\end{cases}\]
(as identities of one-dimensional measures). 
Thus the proposition follows. 
\end{proof}

\begin{rem}
Integration of the identity of Proposition \ref{prop:measures} 
shows the sum formula $S_d(\{1\}^{r-1},2)=S_1(\{1\}^{r-1},2)=\zeta(r+1)$ 
\cite{G}. 
\end{rem}

We extend the transformation \eqref{eq:x to z} to 
the bijection between $\Omega^{(d)}=\Delta^{(d)}\times(0,1)^r$ 
and $\Omega=\Delta\times(0,1)^r$ simply as $(\bx,\bv)\mapsto(\bz,\bv)$. 
By Proposition \ref{prop:measures}, we have 
the identity $\mu^{(d)}=\mu\coloneqq\mu^{(1)}$ of measures 
under this bijection. 

\section{Comparison of $d$ and $d+1$}
Through the transformation \eqref{eq:x to z}, 
we regard the functions on $\Omega^{(d)}$ considered 
in the previous sections as the functions on $\Omega$; 
these functions are denoted by adding the superscript $(d)$. 
For example, the coordinates $x_i$ on $\Omega^{(d)}$ will be written 
as $x^{(d)}_i$ when they are regarded as functions on $\Omega$. 
In addition, the equation \eqref{eq:S_d integral} is rewritten as
\[S_d(\bk)=\int_{\Omega}
(N^{(d)}_1)^{-(k_1-1)}\dots (N^{(d)}_{r-1})^{-(k_{r-1}-1)}
(N^{(d)}_r)^{-(k_r-2)}\mu. \]
Therefore, the proof of Theorem \ref{thm:main} is reduced to the following statement.  

\begin{prop}
For $1\le d<r$, we have $N^{(d)}_l\le N^{(d+1)}_l$ for every $l=1,\ldots,r$, almost everywhere on $\Omega$. 
Moreover, these inequalities are strict simultaneously on a set of positive measure. 
\end{prop}
\begin{proof}
By definition, we have 
\[N^{(d)}_l=\sum_{j=0}^{l-1}\delta^{(d)}_j
=\sum_{\substack{0\le i<d\\ p^{(d)}_i<l}}\lceil\log_{x^{(d)}_i}v_i\rceil, \]
and the same formula holds for $d+1$. Therefore, to prove $N^{(d)}_l\le N^{(d+1)}_l$, 
it suffices to show that  
\[x^{(d)}_i<x^{(d+1)}_i\ \text{ and }\ p^{(d)}_i\ge p^{(d+1)}_i\ \text{ for }\ i=0,\ldots,d-1. \]

By inverting the transformation \eqref{eq:x to z}, we have 
\[x^{(d)}_0=\frac{z_0(1-z_d)}{z_d(1-z_0)}\ \text{ and }\ x^{(d+1)}_0=\frac{z_0(1-z_{d+1})}{z_{d+1}(1-z_0)},\]
which imply 
\[x^{(d+1)}_0-x^{(d)}_0=\frac{z_0}{1-z_0}\biggl(\frac{1}{z_{d+1}}-\frac{1}{z_d}\biggr)>0. \]
Similarly, for $i=1,\ldots,d-1$, the formulas 
\[x^{(d)}_i=1-\frac{z_i(z_d-z_0)}{z_d(1-z_0)}\ \text{ and }\ 
x^{(d+1)}_i=1-\frac{z_i(z_{d+1}-z_0)}{z_{d+1}(1-z_0)}\] 
imply 
\[x^{(d+1)}_i-x^{(d)}_i=\frac{z_i}{1-z_0}\biggl(\frac{z_0}{z_{d+1}}-\frac{z_0}{z_d}\biggr)>0. \] 

Next we show $p^{(d)}_i\ge p^{(d+1)}_i$ for $i=1,\ldots,d-1$ (obviously $p^{(d)}_0=p^{(d+1)}_0=0$). 
By definition, we have 
\begin{align*}
p^{(d)}_i&=i+\#\{j\in\{d+1,\ldots,r\}\mid x^{(d)}_i>x^{(d)}_j\}, \\
p^{(d+1)}_i&=i+\#\{j\in\{d+2,\ldots,r\}\mid x^{(d+1)}_i>x^{(d+1)}_j\}. 
\end{align*}
Since 
\[x^{(d)}_j=\frac{z_j(1-z_d)}{z_d(1-z_0)}, \]
we have 
\begin{align*}
x^{(d)}_i>x^{(d)}_j 
&\iff 1-\frac{z_i(z_d-z_0)}{z_d(1-z_0)}-\frac{z_j(1-z_d)}{z_d(1-z_0)}>0\\
&\iff z_d(1-z_0)-z_i(z_d-z_0)-z_j(1-z_d)>0\\
&\iff z_d(1-z_i+z_j-z_0)>z_j-z_0z_i. 
\end{align*}
A similar equivalence holds for $d+1$. Since $(1-z_i)+(z_j-z_0)>0$ and $z_d>z_{d+1}$, 
the implication 
\[x^{(d+1)}_i>x^{(d+1)}_j \implies x^{(d)}_i>x^{(d)}_j\]
holds for $j=d+2,\ldots,r$. Moreover, $j=d+1$ can only be counted in $p^{(d)}_i$. 
Thus we obtain $p^{(d)}_i\ge p^{(d+1)}_i$, and we complete the proof of $N^{(d)}_l\le N^{(d+1)}_l$. 

For the strictness of the inequality, it suffices to have $\delta^{(d)}_0<\delta^{(d+1)}_0$. 
This holds if, say, $x^{(d)}_0<v_0<x^{(d+1)}_0$, and this condition defines a set of positive measure. 
\end{proof}

\end{document}